\documentclass[12pt]{amsart}
\usepackage{amssymb}
\usepackage{amsmath}
\usepackage[all]{xy}

\input xy
\xyoption{all}

\newcommand\Zset{\mathbb {Z}}
\newcommand\Cset{\mathbb {C}}
\newcommand\Qset{\mathbb {Q}}

\newcommand\Qrmax{Q^r_{\mathrm{max}}}

\newcommand\End{\mathrm{End}}
\newcommand\ann{{\mathrm{ann}}}

\newtheorem{lemma}{Lemma}[section]
\newtheorem{theorem}[lemma]{Theorem}
\newtheorem{corollary}[lemma]{Corollary}
\newtheorem{prop}[lemma]{Proposition}

\theoremstyle{definition}

\newtheorem{example}[lemma]{Example}

\begin{document}

\title{Classes of almost clean rings}

\author{Evrim Akalan}
\address{Department of Mathematics\\ Hacettepe University, Beytepe Campus\\ Ankara, 06532 Turkey}
\email{eakalan@hacettepe.edu.tr}

\author{Lia Va\v s}
\address{Department of Mathematics, Physics and Statistics\\
University of the Sciences\\
Phila\-delphia, PA 19104, USA}
\email{l.vas@usciences.edu}

\thanks{This work was supported by the Visiting Scientists Fellowship Program grant from The Scientific and Technological Research Council of Turkey (TUBITAK) and was carried out in part during the visit of the second author to the  Hacettepe University in July 2011. The second author is grateful to the faculty and staff of Hacettepe University for their hospitality and support. The authors are also grateful to all who helped improve the language of the paper.
}



\begin{abstract}
A ring is clean (almost clean) if each of its elements is the sum of a unit (regular element) and an
idempotent. A module is clean (almost clean) if its endomorphism ring is clean (almost clean). We show that every quasi-continuous and nonsingular module is almost clean and that every right CS (i.e. right extending) and right nonsingular ring is almost clean. As a corollary, all right strongly semihereditary rings, including finite $AW^*$-algebras and noetherian Leavitt path algebras in particular, are almost clean.

We say that a ring $R$ is special clean (special almost clean) if each element $a$ can be decomposed as the sum of a unit  (regular element) $u$ and an idempotent $e$ with $aR\cap eR=0.$ The Camillo-Khurana Theorem characterizes unit-regular rings as special clean rings. We prove an analogous theorem for abelian Rickart rings: an abelian ring is Rickart if and only if it is special almost clean. As a corollary, we show that a right quasi-continuous and right nonsingular ring is left and right Rickart.

If a special (almost) clean decomposition is unique, we say that the ring is uniquely special  (almost) clean. We show that (1) an abelian  ring is unit-regular (equiv. special clean) if and only if it is uniquely special clean, and that (2) an abelian and right quasi-continuous ring is Rickart (equiv. special almost clean) if and only if it is  uniquely special almost clean.

Finally, we adapt some of our results to rings with involution: a $*$-ring is $*$-clean (almost $*$-clean) if each of its elements is the sum of a unit (regular element) and a projection (self-adjoint idempotent). A special (almost) $*$-clean ring is similarly defined by replacing ``idempotent'' with ``projection'' in the appropriate definition. We show that an abelian $*$-ring is a Rickart $*$-ring if and only if it is special almost $*$-clean, and that an abelian $*$-ring is $*$-regular if and only if it is special $*$-clean.
\end{abstract}

\maketitle

\section*{Introduction}

A ring is {\em clean} if each of its elements can be written as the sum
of a unit and an idempotent. W. K. Nicholson introduced the concept of clean rings in the late 1970s. Since then, some stronger concepts (e.g. uniquely clean, strongly clean, and special clean rings) have been considered, as well as some weaker ones (e.g. almost clean rings).

A ring is {\em almost clean} if each of its elements can be written as the sum of a regular element (neither a left nor a right zero-divisor) and an idempotent. Almost clean rings were introduced in \cite{Warren_CX} for commutative rings where it is shown that a commutative Rickart ring is almost clean. In most papers so far, almost cleanness is considered in the commutative case. One of the exceptions is \cite{Lia_clean} which proves that certain Baer $*$-rings (in particular finite $AW^*$-algebras of type $I$) that are not necessarily commutative are almost clean. This result was shown by embedding such a Baer $*$-ring $R$ in the maximal right ring of quotients $\Qrmax(R).$ The ring $\Qrmax(R)$ is unit-regular (thus clean) and has the same projections (self-adjoint idempotents) as $R$. In this situation, the cleanness of $\Qrmax(R)$ implies that the ring $R$ is almost clean.

In this paper, we extend and generalize the idea of \cite{Lia_clean}: we consider a class of rings $R$ with the property that $R$ embeds in a clean ring with the same idempotents as $R$. All right quasi-continuous and right nonsingular rings have this property. We show that they are almost clean (Proposition \ref{quasi_cont_nonsingular}). We also generalize our results to modules (Theorem \ref{quasicont_nonsingular_is_almost_clean}). In Theorem \ref{CS_is_almost_clean}, we show that the assumption that $R$ is right quasi-continuous (C1+C3) can be relaxed to the assumption that $R$ is right CS (right extending, i.e. (C1)).

As a corollary, we also show that the class of right strongly semihereditary rings, studied in \cite{Lia_dimension}, is almost clean (Corollary \ref{AW_Leavitt_etc_corollary}). Consequently, all finite $AW^*$-algebras are almost clean. This fact extends the results in \cite{Lia_clean}: we now know that all finite $AW^*$-algebras are almost clean, not just finite $AW^*$-algebras of type $I.$ In part, this result contributes to determining those von Neumann algebras that are clean (an initiative started by T.Y. Lam). In addition, our result also implies that all noetherian Leavitt path algebras (Leavitt path algebras over finite no-exit graphs) are almost clean.

Clean rings are an additive analogue of unit-regular rings. In a unit-regular ring, each element can be written as the product of a unit and an idempotent. In the case of clean rings, ``the product'' in the last condition changes to ``the sum''. The Camillo-Khurana Theorem in \cite{Camillo} characterizes unit-regular rings as clean rings in which each element $a$ has the form $a=u+e$ where $u$ is a unit and $e$ is an idempotent with $aR\cap eR=0.$ Following the terminology used in \cite{Gene-Ranga}, we refer to the rings satisfying the last condition as {\em special clean} rings.

Our goal is to establish a result analogous to the Camillo-Khurana Theorem: the exact relation between abelian Rickart rings, the rings in which each element can be written as the product of a regular element and an idempotent, and their additive analogues, almost clean rings. We show that an abelian ring is Rickart if and only if each element $a$ has the form $a=u+e$ where $u$ is a regular element and $e$ is an idempotent with $aR\cap eR=0$ (Theorem \ref{Almost_clean_Camillo_Khurana}).  We refer to the rings satisfying the last condition as {\em special almost clean} rings. Interestingly, this result has a corollary that a right quasi-continuous and right nonsingular ring is both left and right Rickart (Corollary \ref{abelian_dropped}). Note that \cite[Theorem 3.2]{Beidar_et_al} demonstrates that a right quasi-continuous and right nonsingular ring $R$ is right Rickart. In this situation, our result shows that $R$ is left Rickart as well.

We show that in the abelian case, the Camillo-Khurana Theorem can be strengthened to state that $R$ is unit-regular if and only if it is {\em uniquely special clean}, i.e. special clean decompositions are unique (Proposition \ref{uniquely_Camillo_Khurana}). As a corollary, we deduce that all abelian,  right quasi-continuous, right nonsingular rings are {\em uniquely special almost clean}, i.e. special almost clean decompositions are unique (Corollary \ref{uniquely_special_almost_clean}). Furthermore, an abelian, right quasi-continuous ring is Rickart if and only if it is uniquely special almost clean (Corollary \ref{uniquely_Rickart_corollary}).

Finally, we turn to $*$-rings and study their cleanness in the context of the presence of an involution. In \cite{Lia_clean}, a ring with involution is said to be {\em $*$-clean} ({\em almost $*$-clean}) if each of its elements is the sum of a unit (regular element) and a projection. We define {\em special $*$-clean} and {\em special almost $*$-clean} rings by replacing ``idempotent'' with ``projection'' in the definitions of special clean and special almost clean, respectively. We show the $*$-version of our characterization of abelian Rickart rings: an abelian $*$-ring is a Rickart $*$-ring if and only if it is special  almost $*$-clean (Theorem \ref{star_almost_clean_Camillo_Khurana}). We also show the $*$-version of the Camillo-Khurana Theorem in the abelian case: an abelian $*$-ring is $*$-regular if and only if it is special $*$-clean (Theorem \ref{star_clean_Camillo_Khurana}).

The paper is organized as follows. In Section \ref{section_preliminaries}, we recall some known concepts and results. In Section \ref{section_almost_clean}, we prove the results related to the almost cleanness of quasi-continuous or CS rings and modules. In Section \ref{section_Rickart}, we prove a theorem on abelian Rickart rings analogous to the Camillo-Khurana Theorem and derive several related corollaries. In Section \ref{section_uniqueness}, we study the uniqueness of special clean and almost special clean decompositions. In Section \ref{section_star_rings}, we adapt our earlier results to rings with involution. We conclude the paper with a list of open problems in Section \ref{section_questions}.

\section{Preliminaries}
\label{section_preliminaries}

In this paper, a ring is an associative ring with identity. We use $\ann^R_r(x)$ and $\ann^R_l(x)$ to denote the right and left annihilators of an element $x$ of a fixed ring $R.$ We use $\ann_r(x)$ and $\ann_l(x)$ when it is clear that these annihilators are in a ring $R$. Throughout the paper, we use the definitions of a regular ring, a unit-regular ring, the maximal right (left) ring of quotients of a ring, an essential submodule and related concepts as given in \cite{Lam}.

First we recall the conditions (C1)--(C3). Let $R$ be a ring and $M$ a right $R$-module. Consider the following three conditions.
\begin{itemize}
\item[(C1)] Every submodule of $M$ is essential inside a summand of $M.$

\item[(C2)] Every submodule of $M$ that is isomorphic to a summand of $M$ is
itself a summand of $M.$

\item[(C3)] If $A$ and $B$ are summands of $M$ with $A \cap B = 0$, then
$A\oplus B$ is also a summand of $M.$
\end{itemize}

These conditions give rise to the following definitions.
\begin{itemize}
\item[-] $M$ is called a {\em CS (or extending)} module if it satisfies (C1).

\item[-] $M$ is called a {\em continuous} module if it satisfies (C1) and (C2).

\item[-] $M$ is called a {\em quasi-continuous (or $\pi$-injective)} module if
it satisfies (C1) and (C3).
\end{itemize}
A ring $R$ is right CS (right quasi-continuous or right continuous) if $R_R$ is CS (quasi-continuous or continuous).

As (C2) implies (C3), a continuous module is quasi-continuous (\cite[Exercise 36, p. 245]{Lam}). Also, if $R$ is regular, then the following are equivalent: $R$ is continuous, $R$ is quasi-continuous and $R$ is CS ( \cite[Exercise 36, p. 246]{Lam}).

The following result of Goel and Jain from \cite{Goel_Jain} can also be found in \cite[Exercise 37, p. 245]{Lam}.

\begin{prop}
\cite{Goel_Jain}
For any module $M$, the following are equivalent.
\begin{enumerate}
\item $M$ is quasi-continuous,

\item Any idempotent endomorphism of a submodule of $M$ extends to an
idempotent endomorphism of $M.$

\item $M$ is invariant under any idempotent endomorphism of the injective envelope $E(M).$
\end{enumerate}
\label{prop_on_quasi-continuous}
\end{prop}

Now, let us turn to the preliminaries on clean rings. A ring element is right (left) regular if it does not have nontrivial right (left) annihilators. It is regular if it is left and right regular. Note that we use the term ``regular'' as in, for example, \cite{Lam} and {\em not} in the sense that $a\in R$ is regular if $a=axa$ for some $x\in R.$

We have reviewed the definitions of (almost) clean, special (almost) clean, and uniquely special (almost) clean rings in the introduction. Using the term ``special clean'', the Camillo-Khurana Theorem can be stated as follows.

\begin{theorem}\cite[Theorem 1]{Camillo} A ring $R$ is unit-regular if and only if it is special clean.
\label{Camillo_Khurana}
\end{theorem}

Recall that an $R$-module $M$ is called (almost) {\em clean} if its endomorphism ring is (almost) clean. In \cite[Theorem 3.9]{Camillo_et_all} one finds the deep result that a continuous module is clean.

Finally, we note the following.
\begin{prop}\cite[Proposition 4.5]{Camillo_et_all}
If $M$ is CS, then each endomorphism of $M$ is the sum of an idempotent
and a monomorphism.
\label{CS_right_almost_clean}
\end{prop}
Since an injective endomorphism of $M$ is a right regular element of the endomorphism ring $\End(M)$, we have the following implications for an arbitrary element $f$ of $\End(M)$.

$$\xymatrix{ \mbox{ CS }\;\; \ar[r] & \;\;f=\mbox{ idempotent }+\mbox{ mono.} \;\;\ar[r] & \;\;f=\mbox{ idempotent }+\mbox{ right regular }}$$

\section{Almost cleanness of quasi-continuous rings and modules}
\label{section_almost_clean}

Proposition \ref{prop_on_quasi-continuous} implies that a ring
$R$ and its
injective envelope $E(R)$ have the same idempotent endomorphisms if and only if
$R$ is right quasi-continuous.
In the case that $R$ is right nonsingular, the maximal right ring of quotients
$\Qrmax(R)$ is the
injective envelope $E(R).$
In this case, $\Qrmax(R)$ and $R$ have the same idempotents if and only if $R$
is right quasi-continuous.

Using an idea from \cite[Proposition 8]{Lia_clean}, we prove the following.

\begin{prop} If $R$ embeds in a clean ring that has the same idempotents as $R$,
then $R$ is almost clean. If $R$ is also regular, then $R$ is clean.
\label{embedding_prop}
\end{prop}
\begin{proof}
Let $Q$ denote a clean ring with the same idempotents as $R$ in which $R$ embeds. Let $a$ be an
element in $R$.
Then $a$ is in $Q$ as well. Thus, $a=u+e$ for some idempotent $e\in Q$ and unit
$u\in Q.$ By assumption, $e$ is in $R$. Thus, $u=a-e$ is in $R$ as well. Since $u$ is
a unit in $Q,$
$0=\ann^Q_r(u)\supseteq\ann_r^R(u)$ and the same holds for the left
annihilators.
Thus, $u$ is regular.

The last sentence of the proposition follows since each regular element of a von Neumann regular ring is a unit.
\end{proof}

\begin{example}
We note that Proposition \ref{embedding_prop} is not valid if ``almost'' is
deleted since $\Zset$ can be embedded in $\Qset$ with the same idempotents (0 and 1) and
$\Zset$ is not clean.
\end{example}

\begin{prop}
If $R$ is right quasi-continuous and right nonsingular, then $R$ is almost
clean.
\label{quasi_cont_nonsingular}
\end{prop}
\begin{proof} If $R$ is right nonsingular, $\Qrmax(R)$ is regular and right
self-injective.
Thus $\Qrmax(R)$ is clean by \cite[Corollary 3.12]{Camillo_et_all}.
Then $R$ is almost clean by Proposition \ref{embedding_prop}.
\end{proof}

In the case when $\Qrmax(R)$ of a right quasi-continuous ring $R$ is unit-regular, a stronger
conclusion than Proposition \ref{quasi_cont_nonsingular} holds as the following corollary shows.
\begin{corollary}
If $R$ is right quasi-continuous and $\Qrmax(R)$ is unit-regular, then $R$ is
special almost clean.
\label{Qunit-reg}
\end{corollary}
\begin{proof}
Since $\Qrmax(R)$ is unit-regular, it is special clean by Theorem \ref{Camillo_Khurana}. Let $a=u+e$
be a special clean decomposition in $\Qrmax(R)$ of an element $a$ of $R.$ Thus, $a\Qrmax(R)\cap e\Qrmax(R)=0.$ Then $aR\cap eR=0$ as well. The idempotent $e$ is in $R$ because $R$ is quasi-continuous. Thus, $u$ has to be in $R$ as well and it has to be regular just like in the proof of Proposition \ref{embedding_prop}. Thus, $a=u+e$ is a special almost clean decomposition of $a$ in $R.$
\end{proof}

Next we consider quasi-continuous modules and prove a stronger
version of Proposition \ref{quasi_cont_nonsingular} -- we prove
that it holds for modules as well. First, we show a preliminary
proposition. Recall that an endomorphism of a module $M$ is said to be essential if its image is an essential
submodule of $M$.

\begin{prop}
If $M$ is quasi-continuous, then every endomorphism of $M$ is the sum of an
idempotent and an essential monomorphism.
\label{first_arrow}
\end{prop}
\begin{proof}
Let $f$ be an endomorphism of $M.$ Then $f$ can be extended to the injective
envelope $E(M)$ of $M.$ Let $\overline{f}$ denote this extension.

The module $E(M)$ is injective and, therefore clean by \cite[Corollary
3.11]{Camillo_et_all}.
There exist an idempotent $\overline{e}$ and a unit  $\overline{u}$ in the ring of
endomorphism of $E(M)$ such that
$\overline{f}=\overline{e}+\overline{u}$. The restriction $e$ of the idempotent $\overline{e}$ to $M$ is in $\End(M)$ since $M$ is
quasi-continuous. If $u$ denotes the
restriction of $\overline{u}$ to $M,$ it is clearly a monomorphism. We claim that
the image $u(M)$ is essential in $M.$

First, note that $u$ is $M$-invariant since $u=f-e.$ Then, note that the fact
that $M$ being essential
in $E(M)$ implies that $u(M)$ is essential in $\overline{u}(E(M))$ since $u$ is
a monomorphism.
In addition, $\overline{u}$ is a unit and so it is onto. Thus,
$\overline{u}(E(M))=E(M).$
This shows that $u(M)$ is essential in $E(M)$ and therefore essential in $M$ as well.
\end{proof}

\begin{theorem}
If $M$ is a quasi-continuous and nonsingular module, then $M$ is almost clean.
\label{quasicont_nonsingular_is_almost_clean}
\end{theorem}
\begin{proof}
Let $M$ be a quasi-continuous, nonsingular module and $f$ be an endomorphism of $M.$
Using
Proposition \ref{first_arrow}, $f$ can be written as $e+u$ where $e$ is an
idempotent and $u$ is an
essential monomorphism. Hence, $u$ is right regular in $\End(M)$. We need to prove that $u$ is
left regular as well.
Let us assume that $gu=0$ for some endomorphism $g$ of $M$. Then, the kernel of
$g$ contains the
image $u(M)$ of $u.$ Since $u(M)$ is essential in $M$, $\ker g$ is essential in
$M.$
Therefore, the module $M/\ker g$ is singular (\cite[Example 7.6 (3)]{Lam}). The map
$g$ factors to a monomorphism from the singular module  $M/\ker g$ to the nonsingular module $M$. Hence,
this map has to be zero (\cite[Exercise 4, p. 269]{Lam}). Then $g$ is zero as well.
\end{proof}

We can represent these results as additions to the diagram from the previous section as follows. In the diagram below, the arrows indicate implications. The first column refers to the properties of an $R$-module $M$ and the second two refer to the properties of an arbitrary element $f$ of $\End(M).$

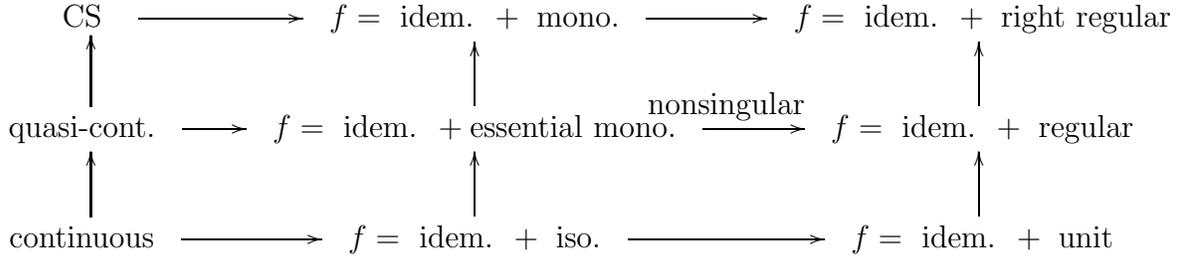
$$\xymatrix{ \mbox{ CS }\;\; \ar[r] & \;\;f=\mbox{ idem. }+\mbox{ mono.} \;\;\ar[r] & \;\;f=\mbox{ idem. }+\mbox{ right regular }\\
\mbox{quasi-cont.}\;\; \ar[r]\ar[u] & \;\;f=\mbox{ idem. }+\mbox{essential mono.} \;\;\ar[r]^{\mbox{nonsingular}}\ar[u] & \;\;f=\mbox{ idem. }+\mbox{ regular }\ar[u]\\
\mbox{continuous}\;\; \ar[r]\ar[u] & \;\;f=\mbox{ idem. }+\mbox{ iso.} \;\;\ar[r]\ar[u] & \;\;f=\mbox{ idem. }+\mbox{ unit }\ar[u]}
$$

The first row is the diagram in Section \ref{section_preliminaries}.
The implications in the second row follow from Proposition \ref{first_arrow} and Theorem \ref{quasicont_nonsingular_is_almost_clean}. The first implication in the third row follows from \cite[Lemma 3.14]{Mohamed_Muller}
stating that a quasi-continuous module $M$ is continuous if and only if each essential monomorphism in $\End(M)$ is an isomorphism. The second implication in the third row is trivial.

Now we show that Theorem \ref{quasicont_nonsingular_is_almost_clean} can be used to strengthen Proposition \ref{quasi_cont_nonsingular}.

\begin{theorem}
If $R$ is a right CS, right nonsingular ring, then each element of $R$ is the sum of an idempotent and a regular element, i.e. $R$ is almost clean.
\label{CS_is_almost_clean}
\end{theorem}
\begin{proof}
Since the endomorphism ring $\End(R)$ is isomorphic to $R$ by $a\mapsto L_a$ where $L_a$ stands for the left multiplication by $a,$ Proposition \ref{CS_right_almost_clean} tells us that for each ring element $a,$ $L_a$ is the sum of an idempotent endomorphism and a monomorphism. Let $L_e$ be the idempotent endomorphism and $L_r$ the monomorphism with $L_a=L_e+L_r.$ Then clearly $a=e+r.$ Also,
$e$ is an idempotent element of $R$ and for every $x\in R,$ $L_r(x)=rx=0$ implies that $x=0$ so $r$ is right regular. We claim that $r$ is left regular as well.

First, note that $rR$ is essential in $R$. Indeed, for every nonzero $x\in R,$ $rx$ is a nonzero element of $rR.$ Assume that $xr=0$ for some $x\in R.$ Then $rR$ is contained in the right annihilator $\ann_r(x).$ Thus, $\ann_r(x)$ is essential in $R$ since $rR$ is essential in $R.$ However, if $R$ is right nonsingular, this implies that $x=0.$ So, $r$ does not have a nontrivial left zero divisor as well. Hence, $r$ is regular.
\end{proof}

Thus, in the ring case, the three-row diagram above simplifies to the following. Here $a$ denotes an arbitrary element of a ring $R.$

$$\xymatrix{ \mbox{ right CS }\;\; & \ar[r]^{\mbox{right nonsingular}} && \;\;\; a=\mbox{ idempotent }+\mbox{ regular }\\
\mbox{ right continuous}\;\;\ar[u] &  \ar[r]\;\;\;\; &&  \;\;\;a=\mbox{ idempotent }+\mbox{ unit }\ar[u]}$$

\begin{example}
The converse of Theorem \ref{CS_is_almost_clean} does not hold. In \cite[Example 5.6]{Chatters_and_Hajarnavis}, it is observed that the ring  $R=\left[\begin{array}{cc} S & S \\ 0 & S\end{array}\right]$ where  $S=\left[\begin{array}{cc} K & K \\ 0 & K\end{array}\right]$ and $K$ is a field, is artinian and not right CS. Since a right artinian ring is clean, $R$ is almost clean. Note also that the ring $S$ is an example of a clean ring that is right CS, right nonsingular, and not right quasi-continuous (see example following Corollary 4.6 in \cite{Camillo_et_all}).

In addition, $\Zset/4\Zset$ is an example of a quasi-continuous, clean ring (since it is self-injective) that is not right nonsingular.
\end{example}

Theorem \ref{CS_is_almost_clean} implies that the class of almost regular rings considered in \cite{Lia_clean} can be widened. Moreover, this theorem proves that the class of rings considered in \cite{Lia_dimension} is also almost clean. Let us elaborate.

In \cite{Lia_dimension}, a ring is said to be right strongly semihereditary
if any of the ten equivalent conditions in \cite[Proposition 3.1]{Lia_dimension} holds. In particular, condition (6) states that
a right strongly semihereditary ring is a right nonsingular ring $R$ such that $R^n$ is CS (as a right $R$-module) for every $n.$ Thus, right strongly semihereditary rings are right nonsingular and right CS and, hence, almost clean.

A ring is strongly semihereditary if it is both left and right strongly semihereditary. By \cite[Examples 3.2 and 4.4]{Lia_dimension}, the following classes of rings and algebras are strongly semihereditary.

\begin{enumerate}
\item A commutative semihereditary and noetherian ring.

\item A finite $AW^*$-algebra (in particular, a finite von Neumann algebra). More generally, Baer $*$-rings satisfying axioms (A1)--(A7) in \cite{Lia_dimension} or \cite{Lia_Baer} are strongly semihereditary (see \cite[Corollary 6.4]{Lia_dimension}).

\item  A Leavitt path algebra over a finite and no-exit graph.
\end{enumerate}

These classes provide different examples of right nonsingular and right CS rings: a finite $AW^*$-algebra is not necessarily right noetherian nor right hereditary; neither a finite $AW^*$-algebra nor a Leavitt path algebra over a finite and no-exit graph is necessarily commutative (for more details, see \cite[Examples 3.2 and 4.4]{Lia_dimension}). Moreover, \cite[Proposition 4.3 and part (3) of Example 4.4]{Lia_dimension} demonstrate that the $2\times 2$ matrix algebra $M_2(K[x,x^{-1}])$ over the Laurent polynomial
ring $K[x,x^{-1}]$ for any positive definite field $K$ is a strongly semihereditary ring that is not
quasi-continuous. Let use examine this example in more detail.

\begin{example}
Let $R$ be $M_2(K[x,x^{-1}])$ where $K$ is a positive definite field (e.g. $\Cset$ with conjugate-complex involution). The maximal right (and left) ring of quotients of $R$
is the ring $Q=M_2(K(x)),$ where $K(x)$ is the field of rational functions over $K.$
By \cite[Part (3) of Example 4.4]{Lia_dimension} the ring $R$ is
a strongly semihereditary ring (thus CS) such that $Q$ has more projections than $R$.
Since projections are idempotents, $R$ is not quasi-continuous by Proposition
\ref{prop_on_quasi-continuous}. Thus, (C1) holds but (C3) does not.

Note also that $R$ is an almost clean ring (by Theorem \ref{CS_is_almost_clean}) that is not clean (since it is not an exchange ring by \cite[Theorem 4.5]{Exchange_LPAs}).
\end{example}

In \cite{Lia_clean}, it is shown that all finite $AW^*$-algebras of type $I$ are almost clean. Using Theorem \ref{CS_is_almost_clean} we can improve this result and state that {\em all} finite $AW^*$-algebras are almost clean.

\begin{corollary}
A right strongly semihereditary ring is almost clean. In particular, a finite $AW^*$-algebra and a Leavitt path algebra over a finite no-exit graph are almost clean.
\label{AW_Leavitt_etc_corollary}
\end{corollary}

We conclude this section with the following observation. By \cite[Corollary 4.8]{Camillo_et_all}, a module $M,$ such that any direct sum of any number of copies of $M$ is CS, is clean. By our Corollary \ref{AW_Leavitt_etc_corollary}, a right nonsingular ring, such that any {\em finite} direct sum of any number of copies of $R$ is CS, is almost clean. The last statement cannot be strengthened to state that $R$ is clean. Namely, $\Zset$ is an example of a strongly semihereditary ring (thus, it is right nonsingular and any finite direct sum of any number of copies of $\Zset$ is CS) which is not clean.

\section{Almost clean characterization of abelian Rickart rings}
\label{section_Rickart}

Recall that clean rings are additive analogues of unit-regular rings:
each element is the
sum (product) of a unit and an idempotent. We illustrate that in
the abelian case, almost clean rings are additive analogues of Rickart rings: each element is the
sum (product) of a regular element and an idempotent.

In \cite[Proposition 16]{Warren_CX}, it is shown  that a commutative
Rickart ring is almost clean. The proof uses \cite[Lemmas 2 and 3]{Endo}. These two lemmas ensure that
an element in an abelian right Rickart ring is the product of a regular element and an idempotent. Moreover, such a ring is also left Rickart. The proof of \cite[Proposition 16]{Warren_CX} formulated for commutative rings uses only the fact that the idempotents are central, not that the ring has to be commutative. Thus, any abelian Rickart ring is almost clean. Our next theorem addresses the converse of this statement and parallels Theorem \ref{Camillo_Khurana}, as the diagram below illustrates. In the next diagram, $a$ denotes an arbitrary element of a ring, and $e$, $r$, and $u$ stand for an idempotent, a regular element, and a unit, respectively, which exist in appropriate situations.

$$
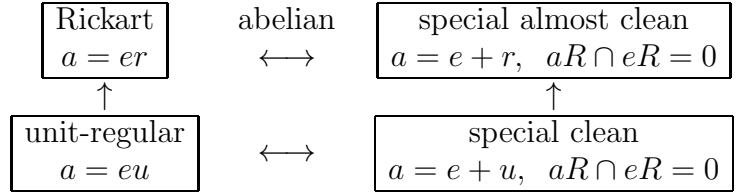
\begin{array}{ccc}
\begin{array}{|c|}\hline
\mbox{Rickart}\\a=er\\ \hline
\end{array}
&
\begin{array}{c}
\mbox{abelian}\\\longleftrightarrow
\end{array}
&
\begin{array}{|c|}\hline
\mbox{special almost clean}\\a=e+r,\;\;aR\cap eR=0\\ \hline
\end{array}
\\
\uparrow & & \uparrow \\
\begin{array}{|c|}\hline
\mbox{unit-regular}\\a=eu\\ \hline
\end{array}
&
\longleftrightarrow
&
\begin{array}{|c|}\hline
\mbox{special clean}\\ a=e+u,\;\;aR\cap eR=0\\ \hline
\end{array}
\end{array}
$$

\begin{theorem} Let $R$ be an abelian ring. Then $R$ is Rickart
if and only if $R$ is special almost clean.
\label{Almost_clean_Camillo_Khurana}
\end{theorem}
\begin{proof}  An abelian ring is right Rickart if and only if it is left Rickart (\cite[Proposition 2]{Endo}). Moreover,
in an abelian Rickart ring, each right
or left regular element is regular.

($\Rightarrow$) If $a$ is an arbitrary element of $R,$ then $\ann_r(a)=eR$
for some idempotent element $e\in R$. In this case, $a=e+a-e$. We claim that
$a-e$ is a regular element of $R.$

To prove this, let $(a-e)r=0$ for $r\in R.$ Since $e\in\ann_r(a),$ $ae=0$
and so $a(1-e)=a.$ Thus, $0=(1-e)(a-e)r=(1-e)ar=a(1-e)r=ar$ which implies
that $r\in \ann_r(a)=eR$. On the other hand, we have $0=e(a-e)r=ear-er=aer-er=-er$ and so
$(1-e)r=r-er=r.$ Thus, $r\in (1-e)R$. Therefore, $r\in eR\cap
(1-e)R=0.$ This shows that $a-e$ is a right regular element of $R$. Since $R$ is an
abelian Rickart ring, $a-e$ is left regular as well.

We also claim that $aR\cap eR=0$. Let $x\in aR\cap eR$. Then
$x=ar=er'$ for some $r, r'\in R$. We have $xe=are=aer=0r=0$ and
$xe=er'e=eer'=er'=x.$
Hence $xe=0=x.$

($\Leftarrow$) Conversely, suppose that $a=e+r$ such that $aR\cap eR=0$
where $e$ is an idempotent and $r$ is a regular element in $R$. We claim that $\ann_r(a)=eR.$

Let $x\in \ann_r(a)$, then $0=ax=ex+rx$ and so $0=eex+erx=ex+erx.$ Thus, we have
that $ex+erx=ex+rx,$ which
implies that $erx=rx.$ Since $R$ is abelian, $rex=erx=rx.$ Hence, $x=ex$ since $r$ is regular. Thus, we have $x\in eR$ and so
$\ann_r(a)\subseteq eR$.

Now, let $x\in eR$. Then $x=ey$ for some $y\in
R$. We have $aey=ey+rey=e(y+ry)\in aR\cap eR=0$, thus $0=aey=ax$ and
so $x\in \ann_r(a)$. Therefore, $\ann_r(a)=eR$ showing that $R$ is right Rickart. Then $R$ is left Rickart as well since $R$ is abelian.
\end{proof}

The following example shows that the assumption that $R$ is abelian cannot be completely eliminated from Theorem  \ref{Almost_clean_Camillo_Khurana}. The example also shows that ``right Rickart'' and ``almost clean'' are independent and exhibits an almost clean Rickart ring that is not special almost clean.

\begin{example}\label{Rickart_example}
Let $R$ be a regular ring that is not clean (for example, we can take Bergman's example, \cite{Bergmans_example}). Then, $R$ is Rickart. Since a regular ring is clean if and only if it is almost clean, $R$ is not almost clean. Thus $R$ is not special almost clean as well.

The ring $\Zset/4\Zset$ is an example of a clean ring that is not right Rickart since it is not right nonsingular. It is not special almost clean since the only clean decomposition of 2, 2=1+1, is not special almost clean. Thus, it is an almost clean ring that is not special almost clean.

Lastly, consider the endomorphism ring $S$ of a countably infinite dimensional vector space over a division ring. The ring $S$ is clean and regular but not unit-regular by \cite[Corollary, page 61]{Nicholson_Varadarajan}. Thus, it is a Rickart, clean ring that is not special clean. Since a regular ring is special clean if and only if it is special almost clean, and since $S$ is regular, $S$ is not special almost clean.
\end{example}

We point out the following corollary of Theorem
\ref{Almost_clean_Camillo_Khurana}.

\begin{corollary} Let $R$ be an abelian and right quasi-continuous ring. The following are equivalent.
\begin{enumerate}
\item $R$ is right nonsingular,
\item $\Qrmax(R)$ is unit-regular,
\item $R$ is Rickart.
\end{enumerate}
\label{Rickart_corollary}
\end{corollary}
\begin{proof}
Note that if $R$ is abelian, $\Qrmax(R)$ is abelian as
well (\cite[Exercise 5, p. 380]{Lam}).

(1) implies (2).  If $R$ is right nonsingular, then $\Qrmax(R)$ is regular.
Thus, $\Qrmax(R)$ is regular and abelian and so it is unit-regular
(\cite[Corollary 4.2]{Goodearl}).

(2) implies (3). If $\Qrmax(R)$ is unit-regular, then $R$ is special almost clean by Corollary \ref{Qunit-reg}. Thus, $R$ is  Rickart by Theorem \ref{Almost_clean_Camillo_Khurana}.

(3) trivially implies (1).
\end{proof}

The following result shows that it is not necessary to assume that $R$ is abelian in order for the equivalence of (1) and (3) in Corollary \ref{Rickart_corollary} to hold.

\begin{corollary} Let $R$ be a right quasi-continuous, right nonsingular
ring, then $R$ is Rickart.
\label{abelian_dropped}
\end{corollary}
\begin{proof} A right quasi-continuous nonsingular
ring can be decomposed as $R=R_1\times R_2$ where $R_1$ is a regular and right
self-injective ring and $R_2$ is a reduced and right quasi-continuous ring by \cite[Proposition 2.7 and Corollary 3.13]{Mohamed_Muller}.
The ring $R_1$ is Rickart since it is regular. The ring
$R_2$ is abelian and nonsingular since it is reduced (\cite[Lemma 7.8, p. 249]{Lam}).
Thus, $R_2$ is an abelian, right
quasi-continuous ring and a right nonsingular ring, and so $R_2$ is Rickart by Corollary \ref{Rickart_corollary}. Hence, $R$ is Rickart as well.
\end{proof}

This corollary parallels the statement that a right nonsingular and right continuous ring is regular (see \cite[Proposition 3.5]{Mohamed_Muller}) as illustrated by the following diagram.

$$\xymatrix{\mbox{ right nonsingular } \;\;\;\;\;\; +&\mbox{ right continuous }\ar[d]&& \ar[r] && \mbox{ regular }\ar[d]\\
\mbox{ right nonsingular }\;\;\;\;\;\;  + &\mbox{ right quasi-continuous }&& \ar[r] && \mbox{ Rickart }
}$$

The fact that a left and right nonsingular, left and right CS ring, is Rickart is already known (see \cite[Theorem 5.1]{Chatters_and_Hajarnavis}). Also, it has been shown that a right nonsingular and right CS ring is right Rickart (\cite[Theorem 3.2]{Beidar_et_al}). Corollary \ref{abelian_dropped} proves that a right nonsingular and right quasi-continuous ring is left Rickart as well.

\section{Uniquely special almost clean decomposition}
\label{section_uniqueness}

In this section, we address the question of uniqueness of special clean and almost clean
decompositions.

\begin{prop}
If $R$ is abelian, the following are equivalent.
\begin{enumerate}
\item $R$ is unit-regular (equivalently, special clean).

\item $R$ is uniquely special clean.
\end{enumerate}
\label{uniquely_Camillo_Khurana}
\end{prop}
\begin{proof}
By the Camillo-Khurana Theorem,
we just need to prove that a special clean ring is uniquely special clean.
Assume that an element $a$ of $R$ has two
special clean decompositions,
$a=e+u$ and $a=e'+u'$ where $e$ and $e'$ are idempotents with $aR\cap eR=0$ and
$aR\cap e'R=0,$ and $u$
and $u'$ are units. Multiplying the relation $a=e+u$ from the right first by $u^{-1}$
and then by $1-e,$
and using that the idempotents are central, we obtain that $au^{-1}(1-e)=1-e.$
Thus, $1-e$ is in $aR.$ Using that $R$ is abelian
again, we obtain
that $(1-e)e'$ is an element both in $e'R$ and $aR.$ So, $(1-e)e'=0,$ implying
that $e'=ee'.$
Relying on the same argument, we obtain that $(1-e')e=0.$ Thus, $e=e'e.$
Hence $e=e'e=ee'=e'.$ Then $u=a-e=a-e'=u'.$
\end{proof}

This proposition has the following corollary.
\begin{corollary} If $R$ is an abelian, right nonsingular, right quasi-continuous
ring, then it is uniquely special almost clean.
\label{uniquely_special_almost_clean}
\end{corollary}
\begin{proof}
Under the assumptions, $Q=\Qrmax(R)$ is unit-regular and thus uniquely special
clean by Proposition \ref{uniquely_Camillo_Khurana}. Furthermore, $R$ is special almost clean by Corollary \ref{Qunit-reg}. We need to show that the special almost clean decomposition in $R$ is unique. Let $a=e+r=e'+r'$ where
$e$ and $e'$ are idempotents with $aR\cap eR=0$ and $aR\cap e'R=0$ and $r$ and $r'$ are regular
elements.

We claim that $r$ and $r'$ are units in $Q.$ The ring $Q$ is unit-regular so $r=uf$ for some unit $u$ and idempotent $f.$
The idempotent $f$ is in $R$ by right quasi-continuity. Then $r(1-f)=uf(1-f)=0.$ By regularity of $r$ in $R$, $1-f=0$ and so $f=1.$ Thus, $r=u$ is a unit in $Q.$

Hence, $a=e+r$ and $a=e'+r'$ are clean decompositions of $a$ in $Q$. We claim that these decompositions are special
clean in $Q.$ Namely, note that $ae=0$ since  $ae\in aR\cap eR.$ If  $aq=eq'\in aQ\cap eQ,$ then $0=eaq=eeq'=eq'.$ So, $aq=eq'=0.$
Since $Q$ is uniquely special clean, $e=e'$ and $r=r'$ which proves that $R$ is uniquely special almost clean.
\end{proof}

This proposition and Theorem \ref{Almost_clean_Camillo_Khurana} have the following corollary,

\begin{corollary}
If $R$ is an abelian and right quasi-continuous ring, the following  condition is equivalent to (1)--(3) of Corollary \ref{Rickart_corollary}.
\begin{itemize}
\item[(4)] $R$ is uniquely special almost clean.
\end{itemize}
\label{uniquely_Rickart_corollary}
\end{corollary}
\begin{proof}
Condition (4) implies that $R$ is special almost clean. This implies condition (3) of Corollary \ref{Rickart_corollary}. Condition (1) implies (4) by Corollary \ref{uniquely_special_almost_clean}.
\end{proof}

We finish this section with the following observations. By \cite[Theorem 3.15]{Lee-Rizvi-Roman}, an $R$-module $M$ is Rickart and (C2)
if and only if the endomorphism
ring $\End(M)$ is regular. In particular, a ring is right Rickart and right (C2)
if and only if it is regular
(\cite[Corollary 3.18]{Lee-Rizvi-Roman}).
In the case that $R$ is a right CS ring, this implies that $R$ is
right Rickart and right continuous if and only if it is regular. Thus, for a right CS and right Rickart
ring, continuity and regularity are equivalent.

Recall that a ring is right morphic if $\ann_r(x)\cong R/xR$ for each $x\in R.$
The following five conditions are equivalent (see \cite[Exercise 19A, p. 270]{Lam} and
\cite[Corollary 3.16]{Zhu_Ding}): (1) $R$ is unit-regular, (2) $R$ is regular and right morphic,
(3) $R$ is regular and left morphic, (4) $R$ is right Rickart and left morphic, and
(5) $R$ is left Rickart and right morphic. Thus, if $R$ is both left and right Rickart,
the conditions of being unit-regular, left morphic and right morphic are equivalent.

Using these facts, we show the following.

\begin{prop} Let $R$ be an abelian, right quasi-continuous and right nonsingular
ring. The following are equivalent.

\begin{enumerate}
\item $R$ is right (left) continuous.

\item $R$ is regular.

\item $R$ is unit-regular (special clean).

\item $R$ is uniquely special clean.

\item $R$ is right (left) morphic.
\end{enumerate}
\label{equivalences}
\end{prop}
\begin{proof}
Note that $R$ is Rickart by Corollary \ref{Rickart_corollary}. Then (1) and (2) are equivalent by  our discussion above. Conditions (2) and (3) are equivalent since $R$ is abelian. Conditions (3) and (4) are equivalent by Proposition \ref{uniquely_Camillo_Khurana}. Finally, conditions (3) and (5) are equivalent by our discussion above.
\end{proof}

\section{Almost cleanness of involutive rings}
\label{section_star_rings}

If $R$ is a ring with involution $*$ (an additive map on $R$ with $(ab)^*=b^*a^*$ and $(a^*)^*=a$ for all $a,b\in R$), it is more natural to
work with projections, self-adjoint idempotents, than idempotents. For example, for rings with involution, the properties of being $*$-regular, Baer $*$-ring or Rickart $*$-ring take over the roles of regular, Baer or (right or left) Rickart, respectively. In \cite{Lia_clean}, the concepts of clean and almost clean rings are adapted to $*$-rings to utilize the presence of an involution: a $*$-ring is {\em $*$-clean} ({\em almost $*$-clean)} if each of its elements is
the sum of a unit (regular element) and a projection. Analogously, we define special $*$-clean and special almost $*$-clean rings by  replacing ``idempotent'' with ``projection'' in the definitions of special clean and special almost clean.
In this section, we adapt some of our earlier results to involutive rings.

First, we have the $*$-version of Proposition \ref{embedding_prop}.

\begin{prop} If a $*$-ring $R$ embeds in a (special) $*$-clean ring that has the same projections as
$R$, then $R$ is (special) almost $*$-clean. If $R$ is also regular, then $R$ is
(special) $*$-clean.
\label{star_embedding_prop}
\end{prop}
\begin{proof} The statement without ``special'' in the three places is \cite[Proposition 8]{Lia_clean}. Let us consider the statement with ``special'' present.

Let $R$ embed into a special $*$-clean ring $Q$ with the same projections as $R.$ Let $a$ be in $R$ and let  $a=p+u$ be a special $*$-clean decomposition of $a$ in $Q.$ Then, $aR\cap pR\subseteq aQ\cap pQ=0$ and $u=a-p$ is a regular element in $R.$ So $R$ is special almost $*$-clean. If $R$ is also regular, then $u$ is a unit in $R$ as well and therefore $R$ is special $*$-clean.
\end{proof}

We prove the $*$-version of Theorem \ref{Almost_clean_Camillo_Khurana} now.

\begin{theorem} Let $R$ be an abelian $*$-ring. Then $R$ is a Rickart
$*$-ring if and only if $R$ is special almost $*$-clean.
\label{star_almost_clean_Camillo_Khurana}
\end{theorem}
\begin{proof} Note that the definition of a Rickart $*$-ring is left-right
symmetric.

If $R$ is an abelian Rickart $*$-ring, then every idempotent is a projection (\cite[Lemma 3]{Lia_clean}). The ring $R$ is special almost clean by Theorem \ref{Almost_clean_Camillo_Khurana}. Hence, it is also special almost $*$-clean.

The converse follows from the proof of Theorem
\ref{Almost_clean_Camillo_Khurana} by assuming that the idempotent $e$ in the proof
of direction  $\Leftarrow$ of Theorem \ref{Almost_clean_Camillo_Khurana} is a
projection.
\end{proof}

Proposition 6 in \cite{Lia_clean}
states the $*$-version of one direction in the Camillo-Khurana Theorem for abelian $*$-rings: if $R$ is
$*$-regular and abelian, then it is special $*$-clean. Interestingly, the almost $*$-clean characterization of abelian Rickart $*$-rings gives us the converse in the abelian case.
\begin{theorem}
Let $R$ be an abelian $*$-ring. Then $R$ is $*$-regular if and only if $R$ is special $*$-clean.
\label{star_clean_Camillo_Khurana}
\end{theorem}
\begin{proof} A $*$-regular and abelian $*$-ring is special $*$-clean by \cite[Proposition 6]{Lia_clean}. Alternatively, this also follows from our Theorem \ref{star_almost_clean_Camillo_Khurana}.  Namely, if $R$ is an abelian $*$-regular ring, then $R$ is a Rickart $*$-ring. Hence, $R$ is special almost $*$-clean. In addition, $R$ is regular and so every regular element of $R$ is a unit. Thus, $R$ is special $*$-clean.

To prove the converse, assume that $R$ is special $*$-clean. Then $R$ is special almost $*$-clean as well and so it is a Rickart $*$-ring by Theorem \ref{star_almost_clean_Camillo_Khurana}. Moreover, $R$ is a special clean ring and so it is unit-regular. A Rickart $*$-ring that is regular is $*$-regular. This proves the assertion.
\end{proof}

We note that the uniqueness of the special (almost) $*$-clean decomposition can be added to Theorems \ref{star_almost_clean_Camillo_Khurana} and \ref{star_clean_Camillo_Khurana}.

\begin{corollary} Let $R$ be an abelian $*$-ring.
\begin{enumerate}
\item If $R$ is $*$-regular, then $R$ is uniquely special $*$-clean.

\item $R$ is Rickart, then $R$ is uniquely special almost $*$-clean.
\end{enumerate}
\label{star_uniqueness}
\end{corollary}
\begin{proof}
In both cases, $R$ is an abelian Rickart $*$-ring and so all the idempotents are projections. Therefore, the two assertions follow from Proposition \ref{uniquely_Camillo_Khurana} and Corollary \ref{uniquely_Rickart_corollary}.
\end{proof}

\section{Questions}
\label{section_questions}

We conclude the paper with a list of questions.
\begin{enumerate}
\item The ring in the Camillo-Khurana Theorem is not assumed to be abelian. We wonder if the assumption that $R$ is abelian in Theorems \ref{Almost_clean_Camillo_Khurana} and
\ref{star_almost_clean_Camillo_Khurana} can be weakened. Example \ref{Rickart_example} shows that it cannot be completely eliminated from Theorem \ref{Almost_clean_Camillo_Khurana}.

We also wonder if the assumption that $R$ is abelian can be weakened in Proposition \ref{uniquely_Camillo_Khurana} and Theorem \ref{star_clean_Camillo_Khurana}.

\item Recall that a quasi-continuous module $M$ is continuous if and only if each essential monomorphism in $\End(M)$ is an isomorphism (\cite[Lemma 3.14]{Mohamed_Muller}). If we denote the last condition by (C), we wonder if (C) is a condition for
an almost clean module to be clean.

$$\begin{array}{ccc}
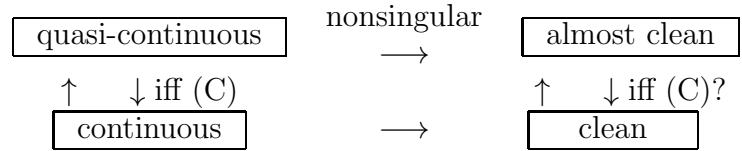
\begin{array}{|c|}\hline
\mbox{ quasi-continuous } \\\hline
\end{array} &
\begin{array}{c}
\mbox{nonsingular}\\
\longrightarrow
\end{array}
&
\begin{array}{|c|}\hline
\mbox{ almost clean } \\\hline
\end{array} \\
\uparrow\;\;\;\;\;\; \downarrow\mbox{iff (C)} & &\uparrow\;\;\;\;\;\; \downarrow\mbox{iff (C)?} \\
\begin{array}{|c|}\hline
\mbox{ continuous } \\\hline
\end{array}
 & \longrightarrow &
\begin{array}{|c|}\hline
\;\;\;\mbox{ clean }\;\;\; \\\hline
\end{array}\;\;\;\;
\end{array}
$$

If (C) guaranteed that almost clean rings are clean, this would imply that a right quasi-continuous, right nonsingular, and clean ring is right continuous. Note that this statement holds if ``clean'' is replaced by ``regular''. More specifically, we wonder whether ``$R$ is clean'' can be added to the list of equivalent conditions in Proposition \ref{equivalences}, i.e., whether an abelian, right quasi-continuous, right nonsingular clean ring is right continuous.

\item All abelian $AW^*$-algebras are uniquely special almost $*$-clean by Corollary \ref{star_uniqueness}. All finite, type $I$ $AW^*$-algebras are almost $*$-clean by \cite[Corollary 14]{Lia_clean}. Finally, all finite $AW^*$-algebras are almost clean by Corollary \ref{AW_Leavitt_etc_corollary}. The diagram below illustrates these statements graphically.
$$
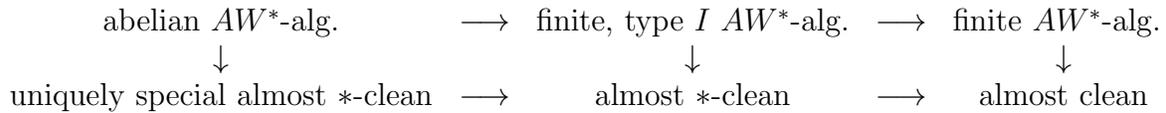
\begin{array}{ccccc}
\mbox{abelian $AW^*$-alg.} & \longrightarrow &  \mbox{finite, type $I$ $AW^*$-alg.} & \longrightarrow & \mbox{finite $AW^*$-alg. }\\
\downarrow & & \downarrow & & \downarrow\\
\mbox{uniquely special almost $*$-clean} & \longrightarrow & \mbox{almost $*$-clean}&\longrightarrow & \mbox{almost clean}
\end{array}
$$

We wonder if any of the statements in the columns of this diagram can be strengthened. In particular, we wonder if finite, type $I$ $AW^*$-algebras are special almost $*$-clean and if finite $AW^*$-algebras are (special) almost $*$-clean. We also wonder if any of these algebras are ($*$-)clean.
\end{enumerate}

\end{document}